\documentclass[11pt]{article}

\usepackage{mathrsfs}
\usepackage{amsfonts}
\usepackage{titlesec}
\usepackage{enumerate}
\usepackage[toc,page,title,titletoc,header]{appendix}
\usepackage{amsmath,amsthm,amssymb,anysize}
\usepackage[numbers,sort&compress]{natbib}
\usepackage{color}
\usepackage{graphicx}

\theoremstyle{definition}
\newtheorem{thm}{Theorem}

\newtheorem{lem}{Lemma}[section]
\newtheorem{rem}[lem]{Remark}

\newtheorem{cor}[lem]{Corollary}

\marginsize{3cm}{3cm}{2.5cm}{2.2cm}
\numberwithin{equation}{section}

\newcommand{\F}{\mathbb{F}}

\newcommand{\Z}{\mathbb{Z}}

\title{\textsf{Adjoint cohomology of two-step nilpotent Lie superalgebras}}
\author{Wende Liu $^{a}$, Yong Yang $^{b,}$ \footnote{Correspondence: yangyong195888221@163.com}\setcounter{footnote}{-1}
\ and Xiankun Du $^{b}$
\\
\\ \small\textit{ $^{a}$School of Mathematics and Statistics, Hainan Normal University, Haikou 571158, China}
\\ \small\textit{$^{b}$School of Mathematics, Jilin University, Changchun 130012, China } }

\date{ }

\begin{document}
\maketitle

\begin{quotation}
\small\noindent \textbf{Abstract}:
 In this paper, we study the cup products and Betti numbers over cohomology superspaces of two-step nilpotent Lie  superalgebras with coefficients in the adjoint modules over an algebraically closed field  of characteristic zero. As an application, we prove that the cup product over the adjoint cohomology superspaces for Heisenberg Lie  superalgebras is trivial  and  we also determine the adjoint Betti numbers for Heisenberg Lie superalgebras by means of  Hochschild-Serre spectral sequences.

\vspace{0.2cm} \noindent{\textbf{Keywords}}: nilpotent Lie superalgebra; cup product; Betti number; spectral sequence

\vspace{0.2cm} \noindent{\textbf{Mathematics Subject Classification 2010}}: 17B30, 17B56

\end{quotation}
\setcounter{section}{-1}
\section{Introduction}
Cohomology of Lie superalgebras was studied by a number of people, including Fuchs and Leites \cite{1}, Fuchs \cite{2}, Scheunert and Zhang \cite{3}, Su and Zhang \cite{4}, Bouarroudj,  Grozman,  Lebedev and Leites \cite{5}, and Musson \cite{6}. As in Lie algebra case, Lie superalgebra cohomology has also many applications in the deformation theory (see \cite{2,6,7,8} for example).
The cup
product, introduced by Musson in \cite{6}, induces an associative superalgebra structure on the trivial cohomology. A standard fact is that the superalgebra structure on the  cochains with coefficients in the trivial modules, which is  arising from the cup product, is isomorphic to the super-exterior algebra generated by the dual superspace of the Lie superalgebra (\cite{2,6}). For some classes nilpotent Lie superalgebras, the cup products and Betti numbers
are studied (\cite{9,10,11}). In particular, Bai and Liu \cite{9} showed that the trivial cohomology for the Heisenberg Lie superalgebra of even center is isomorphic to a quotient algebra of the super-exterior algebra and  the trivial cohomology for the Heisenberg Lie superalgebra of odd center is isomorphic to the direct sum of a quotient algebra of the super-exterior algebra with a trivial-multiplication infinite-dimensional superalgebra.

For a symplectic vector space, one can define its Heisenberg Lie algebra by the symplectic form, which is a two-step nilpotent Lie algebra with one-dimensional center.
Heisenberg Lie algebras have attracted   special attentions in the modern mathematics and physics because of their applications  in the commutation relations in quantum mechanics.
For example,
Santharoubane \cite{12} gave a description of  cocycles, coboundaries and cohomological spaces for Heisenberg Lie algebras with coefficients in the trivial module over a field of characteristic zero. Sk\"{o}ldberg \cite{13} determined the generating function of  Betti numbers of  Heisenberg Lie algebras over a field of characteristic two by means of  discrete Morse theory.
Cairns and  Jambor \cite{14} extended Sk\"{o}ldberg's result to arbitrary  characteristic by directly computing the Betti numbers and showed that in characteristic two, unlike all the other cases, the Betti numbers are unimodal.
The adjoint cohomology of Heisenberg Lie algebras was studied in
the work  of Magnin \cite{15},  Cagliero and Tirao \cite{16} (see also \cite[Remark 1.2 and Example 1]{17}). Moreover, by considering a Heisenberg Lie algebra as the nilradical of a parabolic subalgebra of a simple Lie algebra of type A, Alvarez \cite{18} gave a full description of the adjoint
homology of the parabolic subalgebra as a module over its Levi factor.
Rodr\'{i}guez-Vallarte, Salgado  and  S\'{a}nchez-Valenzuela \cite{19} generalized Heisenberg Lie algebra by considering its supersymmetry, which is called Heisenberg Lie superalgebra, and proved that it admits neither supersymplectic nor superorthogonal invariant forms, however one-dimensional extensions by some appropriate homogeneous derivations do.

In this article, the cup products and Betti numbers on the adjoint cohomology are studied.
For two-step nilpotent Lie superalgebras,
we give a criterion for judging triviality of cup products and describe the Betti numbers by using the Hochschild-Serre spectral sequence.
As an application, for the Heisenberg Lie superalgebras, we prove that the cup products are trivial over the adjoint cohomology which is very different from the result in the trivial cohomology case.
We also depict the Betti numbers of the adjoint cohomology for the Heisenberg Lie superalgebras by using the Hochschild-Serre spectral sequence to the center.

\section{Preliminaries}

Throughout this paper, the ground field $\F$ is an algebraically closed field of characteristic zero and all vector spaces, algebras are over $\F$.

A \emph{Lie superalgebra} is a $\Z_2$-graded algebra whose multiplication satisfies
the skew-supersymmetry and the super Jacobi identity (see \cite{6}).
For a Lie superalgebra $\mathfrak{g}$, we inductively define the lower central series
\begin{equation*}
\mathfrak{g}^0=\mathfrak{g}, \qquad \mathfrak{g}^{i+1}=[\mathfrak{g},\mathfrak{g}^i],\quad i\geq 0.
\end{equation*}
$\mathfrak{g}$ is called $n$-step nilpotent if
\begin{equation*}
\mathfrak{g}^{n-1}\neq 0, \qquad \mathfrak{g}^{n}=0,
\end{equation*}
for some $n\in \mathbb{N}$ (see \cite{19}).

Set  $V=V_{\overline{0}}\bigoplus V_{\overline{1}}$ to be a $\mathbb{Z}_{2}$-graded vector space. Write
 $|v|$ for the $\mathbb{Z}_{2}$-degree for a $\mathbb{Z}_{2}$-homogeneous element $v$ in $V$.
Denote by $T (V)$ the tensor superalgebra generated by $V$. Then $T (V)$ is an associative superalgebra.
Note that $T (V)=\bigoplus\limits_{i\in \mathbb{N}_{0}}\left(\bigotimes^{i}V\right)$ also has a $\mathbb{Z}$-grading structure given by setting $\|v\|=1$ for any $v\in V$.
Hereafter $\| x \|$ denotes the $\mathbb{Z}$-degree of a $\mathbb{Z}$-homogeneous element $x$ in a $\mathbb{Z}$-graded space. 
The super-exterior algebra $\bigwedge^{\bullet} V$ generated by $V$ is defined as the quotient of $T (V)$ by the ideal generated by all elements of the form
$$x\otimes y+(-1)^{|x||y|}y\otimes x,\quad x,y\in V.$$
Then $\bigwedge^{\bullet} V$ is a graded-supercommutative superalgebra (see \cite{6}).
Write $\mathfrak{d}^{k}(r,s)$ for the dimension of $k$-homogeneous subspace of the super-exterior algebra generated by $r$ even generators and $s$ odd generators. Then
$$\mathfrak{d}^{k}(r,s)=\sum_{i=0}^{k}
\left(
                                        \begin{array}{c}
                                          r \\
                                          k-i \\
                                        \end{array}
                                      \right)
\left(
  \begin{array}{c}
    s+i-1 \\
    i \\
  \end{array}
\right).$$

Now, we introduce the definition of the cohomology of Lie superalgebras. For more details, the reader is referred to \cite{2,5,6}. Denote by $\mathfrak{g}$ a Lie superalgebra and $\{x_{i}\mid i\in I\}$ a homogeneous basis of $\mathfrak{g}$, write
$\{x_{i}^{\ast}\mid i\in I\}$
for the dual basis.
For 1-cochains with trivial coefficients, the differential is defined as an operation dual to the Lie-super bracket:
$$\mathrm{d}:\ \mathfrak{g}^{\ast}\longrightarrow   \bigwedge^{2} \mathfrak{g}^{\ast},$$
satisfying that
\begin{equation}\label{z}
\mathrm{d}(x_{i}^{\ast})=\sum_{  k<l }(-1)^{|x_{k}||x_{l}|+1}a_{kl}^{i}x_{k}^{\ast}\wedge x_{l}^{\ast}+\frac{1}{2}\sum_{  k\in I }a_{kk}^{i}x_{k}^{\ast^{2}},\  i\in I,
\end{equation}
where $a_{kl}^{i}$, $ i,k,l\in I$, are the structure constants of $\mathfrak{g}$ with respect to the basis, that is,
$$[x_{k},x_{l}]=\sum_{i\in I}a_{kl}^{i}x_{i},\quad   k,l\in I.$$
Suppose $k\geq 2$. For $k$-cochains with trivial coefficients, $\mathrm{d}$ is defined by the Leibniz rule. That is,
\begin{equation*}
\mathrm{d}(x\wedge y)=\mathrm{d}(x)\wedge y+(-1)^{\|x\|}x\wedge \mathrm{d}(y), \quad x,y\in \bigwedge^{\bullet}\mathfrak{g}^{\ast}.
\end{equation*}
 For cochains  with coefficients in any module $M$, we set
\begin{align} \label{22}
\mathrm{d}(m)=\sum_{i\in I}(-1)^{|x_{i}||m|} (x_{i}\cdot m)\otimes x_{i}^{\ast},
\end{align}
and
\begin{equation*}
 \mathrm{d}(m\otimes w)=\mathrm{d}(m)\wedge w+m\otimes \mathrm{d}(w),
\end{equation*}
for any $m\in M$, $w\in \bigwedge^{\bullet} \mathfrak{g}^{\ast}$  (see \cite[Lemma 3.2]{5}). Note that $M\bigotimes \bigwedge^{\bullet}\mathfrak{g}^{\ast}$ is a $\mathfrak{g}$-module in a natural manner:
$$x\cdot(m\otimes w)=(x\cdot m)\otimes w+(-1)^{|x||m|}m\otimes (x\cdot w),$$
where $x\in \mathfrak{g}$, $m\in M$, and $w\in \bigwedge^{\bullet} \mathfrak{g}^{\ast}$.
We obtain that $\mathrm{d}$ is a $\mathfrak{g}$-module homomorphism and $\mathrm{d}^{2}=0$.
Denote by $\mathrm{H}^{\bullet}(\mathfrak{g},M)$ the \emph{cohomology} of $\mathfrak{g}$ with coefficients in $M$ defined by the cochain complex $(M\bigotimes \bigwedge^{\bullet}\mathfrak{g}^{\ast}, \mathrm{d})$. From Eq. (\ref{22}), we have
\begin{align*}
\mathrm{H}^{0}(\mathfrak{g},M)=\{m\in M\mid \mathfrak{g}\cdot m=0\}.
\end{align*}
Denote by $C(\mathfrak{g})$ the center of $\mathfrak{g}$.
Then
\begin{align}\label{vanish}
\mathrm{H}^{0}(\mathfrak{g},\mathfrak{g})=C(\mathfrak{g}).
\end{align}

\section{Two-step nilpotent Lie superalgebra}
In this section, we study the adjoint cohomology of two-step nilpotent Lie superalgebras.
Let us first recall the definition of the cup product for Lie superalgebras, introduced by Musson in \cite{6}.
Suppose that $S_{n}$ is the symmetric group of degree $n$. For a $\sigma\in S_{n}$, set $\varepsilon(\sigma)$ to be the sign of $\sigma$
and denote the set of inversions of $\sigma$ by
$$\mathrm{Inv}(\sigma)=\{(i,j)|\ i<j\ \mathrm{and}\ \sigma(i)>\sigma(j)\}.$$
Denote by $\mathfrak{g}$ a Lie superalgebra and $M$ a $\mathfrak{g}$-module.
Suppose that $\star: \bigotimes^{2} M\longrightarrow M$, $m_{1}\otimes m_{2}\mapsto m_{1}\star m_{2}$ is a homomorphism of $\mathfrak{g}$-modules.
For $\sigma\in S_{p+q}$, $f\in M\bigotimes \bigwedge^{q} \mathfrak{g}^{\ast}$ and $g\in M\bigotimes \bigwedge^{p} \mathfrak{g}^{\ast}$,
define a bilinear map
\begin{align*}
F_{\sigma}: M\bigotimes \bigwedge^{q} \mathfrak{g}^{\ast} \times M\bigotimes \bigwedge^{p} \mathfrak{g}^{\ast}&\longrightarrow
M\bigotimes \bigwedge^{p+q} \mathfrak{g}^{\ast}
\end{align*}
such that
$$F_{\sigma}(f,g)(x)=(-1)^{|x_{\sigma}^{\uppercase\expandafter{\romannumeral1}}||g|}\varepsilon(\sigma)\gamma(x,\sigma)f(x_{\sigma}^{\uppercase\expandafter{\romannumeral1}}
)\star g(x_{\sigma}^{\uppercase\expandafter{\romannumeral2}}),$$
where $x=(x_{1},\ldots,x_{p+q})\in \mathfrak{g}^{p+q}$, $\gamma(x,\sigma)=\prod\limits_{(i,j)\in\mathrm{Inv}(\sigma)}(-1)^{|x_{\sigma(i)}||x_{\sigma(j)}|}$, $x_{\sigma}^{\uppercase\expandafter{\romannumeral1}}=(x_{\sigma(1)},\ldots, x_{\sigma(q)})$ and $x_{\sigma}^{\uppercase\expandafter{\romannumeral2}}=(x_{\sigma(1+q)},\ldots, x_{\sigma(p+q)})$.
Then the \emph{cup product} of $f$ and $g$ is defined by
$$f\cup g=(1/p!q!)\sum_{\sigma\in S_{p+q}}F_{\sigma}(f,g)$$
(see \cite[16.5.1]{6}). Note that the cup product and the differential are compatible:
$$\mathrm{d}(f\cup g)=\mathrm{d}f\cup g+(-1)^{q}f\cup \mathrm{d}g$$
(see  \cite[Lemma 16.5.5]{6}). Moreover, the cup product induces a $\mathbb{Z}$-graded superalgebra structure on
$M\bigotimes\bigwedge^{\bullet}\mathfrak{g}^{\ast}=\bigoplus\limits_{i\in \mathbb{N}_{0}}\left(M\bigotimes\bigwedge^{i}\mathfrak{g}^{\ast}\right)$
and $\mathrm{H}^{\bullet}(\mathfrak{g},M)=\bigoplus\limits_{i\in \mathbb{N}_{0}}\mathrm{H}^{i}(\mathfrak{g},M)$.
For any $x\in \mathfrak{g}$ and $f\in M\bigotimes \bigwedge^{n+1}\mathfrak{g}^{\ast}$, define $f_{x}\in M\bigotimes \bigwedge^{n}\mathfrak{g}^{\ast}$ by
$$f_{x}(x_{1},\ldots,x_{n})=f(x,x_{1},\ldots,x_{n}),$$
where $x_{1},\ldots,x_{n}\in\mathfrak{g}$. Moreover, from \cite[Lemma 16.5.4]{6}, we obtain the following lemma:
\begin{lem}\label{chain}
For $x\in \mathfrak{g}$ and $f,g\in M\bigotimes \bigwedge^{\bullet}\mathfrak{g}^{\ast}$,
$$(f\cup g)_{x}=(-1)^{|x||g|}f_{x}\cup g+(-1)^{\|f\|}f\cup g_{x}.$$
\end{lem}
\begin{rem}
$(\bigwedge^{\bullet}\mathfrak{g}^{\ast},\bigcup)$ is a graded-supercommutative associative superalgebra (see \cite[Example 16.5.6]{6}). Moreover, it is isomorphic to the super-exterior algebra of $\mathfrak{g}^{\ast}$.
\end{rem}
\begin{thm}
For $f,g,h\in M\bigotimes\bigwedge^{\bullet}\mathfrak{g}^{\ast}$, the following conclusions hold:

(1) if $\star$ satisfies the skew-supersymmetry, then $$g\cup f=(-1)^{|f||g|+\|f\|\|g\|+1}f\cup g,$$

(2) if $\star$ satisfies the super Jacobi identity, then $$f\cup(g\cup h)=(f\cup g)\cup h+(-1)^{|f||g|+\|f\|\|g\|}g\cup (f\cup h).$$
\begin{proof}
(1) Use induction on $\|f\|+\|g\|$. From Lemma \ref{chain} and induction hypothesis, we have
\begin{align*}
(f\cup g)_{x}=&(-1)^{|x||g|}f_{x}\cup g+(-1)^{\|f\|}f\cup g_{x}\\
=&(-1)^{|x||g|}\{(-1)^{|f_{x}||g|+\|f_{x}\|\|g\|+1}g\cup f_{x}\}\\
&+(-1)^{\|f\|}\{(-1)^{|f||g_{x}|+\|f\|\|g_{x}\|+1}g_{x}\cup f\}\\
=&(-1)^{|f||g|+\|f\|\|g\|+1}\{(-1)^{|x||f|}g_{x}\cup f+(-1)^{\|g\|}g\cup f_{x}\}\\
=&(-1)^{|f||g|+\|f\|\|g\|+1}(g\cup f)_{x}.
\end{align*}
The proof of $(2)$ is similar to $(1)$.
\end{proof}
\end{thm}
\begin{rem}
It shows that the cup product on the adjoint cohomology in general is neither unitary nor graded-supercommutative nor associtive, which is different from the trivial cohomology case (see  \cite[Example 16.5.6]{6}, \cite{9} and \cite{10}).
\end{rem}
In particular, $M=\mathfrak{g}$ and $\star$ is the Lie-super bracket. The cup product for the adjoint cohomology can be described by the following theorem.
\begin{thm}\label{cup}
Suppose that $\mathfrak{g}$ is a Lie superalgebra and  $\{x_{i}\mid i\in I\}$ is a basis of $\mathfrak{g}$.
Then the cup product on $\mathfrak{g}\bigotimes\bigwedge^{\bullet}\mathfrak{g}^{\ast}$ is given by
\begin{equation*}
 (x_{i}\otimes f^{\alpha}_{n_{i}})\cup(x_{j}\otimes f^{\beta}_{n_{j}})=(-1)^{|f^{\alpha}_{n_{i}}||x_{j}|}[x_{i},x_{j}]\otimes(f^{\alpha}_{n_{i}}\wedge f^{\beta}_{n_{j}}),
\end{equation*}
where
$$f^{\alpha}_{n_{i}}=x_{\alpha_{1}}^{\ast}\wedge\cdots \wedge x_{\alpha_{n_{i}}}^{\ast}\in\bigwedge^{n_{i}}\mathfrak{g}^{\ast},$$
and
$$f^{\beta}_{n_{j}}=x_{\beta_{1}}^{\ast}\wedge\cdots \wedge x_{\beta_{n_{j}}}^{\ast}\in\bigwedge^{n_{j}}\mathfrak{g}^{\ast},$$
for $n_{i}, n_{j}\in \mathbb{N}_{0}$ and $i,j,\alpha_{1},\ldots, \alpha_{n_{i}},\beta_{1}, \ldots,\beta_{n_{j}}\in I$.
\begin{proof}
Note that
\begin{equation*}
(x_{i}\otimes f^{\alpha}_{n_{i}})(x_{\sigma}^{\uppercase\expandafter{\romannumeral1}}
)\star (x_{j}\otimes f^{\beta}_{n_{j}})(x_{\sigma}^{\uppercase\expandafter{\romannumeral2}})
=f^{\alpha}_{n_{i}}(x_{\sigma}^{\uppercase\expandafter{\romannumeral1}}
)f^{\beta}_{n_{j}}(x_{\sigma}^{\uppercase\expandafter{\romannumeral2}})[x_{i},x_{j}].
\end{equation*}
Moreover, for any $x\in \mathfrak{g}^{n_{i}+n_{j}}$,
\begin{align*}
&\left\{\left(x_{i}\otimes f^{\alpha}_{n_{i}}\right)\cup \left(x_{j}\otimes f^{\beta}_{n_{j}}\right)\right\}(x)\\
=&[x_{i},x_{j}]\left((1/n_{i}!n_{j}!)\sum_{\sigma\in S_{n_{i}+n_{j}}}(-1)^{|x_{\sigma}^{\uppercase\expandafter{\romannumeral1}}||x_{j}\otimes f^{\beta}_{n_{j}}|}\varepsilon(\sigma)\gamma(x,\sigma)
f^{\alpha}_{n_{i}}(x_{\sigma}^{\uppercase\expandafter{\romannumeral1}}
)f^{\beta}_{n_{j}}(x_{\sigma}^{\uppercase\expandafter{\romannumeral2}})\right)\\
=&(-1)^{|f^{\alpha}_{n_{i}}||x_{j}|}[x_{i},x_{j}](f^{\alpha}_{n_{i}}\cup f^{\beta}_{n_{j}})(x)\\
=&(-1)^{|f^{\alpha}_{n_{i}}||x_{j}|}[x_{i},x_{j}](f^{\alpha}_{n_{i}}\wedge f^{\beta}_{n_{j}})(x)\\
=&(-1)^{|f^{\alpha}_{n_{i}}||x_{j}|}\left\{[x_{i},x_{j}]\otimes(f^{\alpha}_{n_{i}}\wedge f^{\beta}_{n_{j}})\right\}(x).
\end{align*}
That is,
$(x_{i}\otimes f^{\alpha}_{n_{i}})\cup(x_{j}\otimes f^{\beta}_{n_{j}})=(-1)^{|f^{\alpha}_{n_{i}}||x_{j}|}[x_{i},x_{j}]\otimes(f^{\alpha}_{n_{i}}\wedge f^{\beta}_{n_{j}})$.
\end{proof}
\end{thm}
\begin{cor}
If $\mathfrak{g}$ is an $n$-step nilpotent Lie superalgebra, then
$$\underbrace{\left(\mathfrak{g}\bigotimes\bigwedge^{\bullet}\mathfrak{g}^{\ast}\right)\bigcup\cdots\bigcup\left(\mathfrak{g}\bigotimes\bigwedge^{\bullet}\mathfrak{g}^{\ast}\right)}_{n+1}=0.$$
\end{cor}
\begin{cor}
If $\mathfrak{g}$ is a Lie superalgebra, then
$$\mathrm{H}^{0}(\mathfrak{g},\mathfrak{g})\bigcup\mathrm{H}^{\bullet}(\mathfrak{g},\mathfrak{g})=0.$$
\end{cor}
\begin{proof}
It can be obtained from Eq. (\ref{vanish}) and Theorem \ref{cup}.
\end{proof}

Moreover, for two-step nilpotent Lie superalgebras,
we give the following criteria for cup product triviality.
\begin{thm}\label{main}
Suppose that $\mathfrak{g}$ is a two-step nilpotent Lie superalgebra and $\mathcal{V}_{\mathfrak{g}}$ is any complement space of $[\mathfrak{g},\mathfrak{g}]$ in $\mathfrak{g}$. Let $\mathcal{V}_{\mathfrak{g}}^{\ast}$ be the dual space of $\mathcal{V}_{\mathfrak{g}}$. For $k\geq 1$, if the following two conditions are satisfied:

(1) $[\mathfrak{g},\mathfrak{g}]\bigotimes \bigwedge^{k}\mathcal{V}_{\mathfrak{g}}^{\ast}\subseteq \mathrm{Im}\ \mathrm{d}$,

(2) For any $f\in[\mathfrak{g},\mathfrak{g}]^{\ast}$ and $g\in\bigwedge^{k}\mathcal{V}_{\mathfrak{g}}^{\ast}$, if $\mathrm{d}(f)\wedge g=0$, then $f=0$ or $g=0$,\\
then the cup product on the adjoint cohomology of $\mathfrak{g}$ is trivial.
\end{thm}
\begin{proof}
For $k\geq 1$, we have
$$\mathfrak{g}\bigotimes \bigwedge^{k}\mathfrak{g}^{\ast}=
\left([\mathfrak{g},\mathfrak{g}]\bigoplus\mathcal{V}_{\mathfrak{g}}\right)\bigotimes\left\{\bigoplus_{i=0}^{k}\left(\bigwedge^{i}[\mathfrak{g},\mathfrak{g}]^{\ast}
\bigwedge^{k-i}\mathcal{V}_{\mathfrak{g}}^{\ast}\right)
\right\}.$$
Since $[\mathfrak{g},\mathfrak{g}]\subseteq C(\mathfrak{g})$, from Eqs. $(\ref{z})$ and $(\ref{22})$, we have
\begin{equation*}
\mathrm{d}([\mathfrak{g},\mathfrak{g}]^{\ast})\subseteq\bigwedge^{2} \mathcal{V}_{\mathfrak{g}}^{\ast},\quad
 \mathrm{d}(\mathcal{V}_{\mathfrak{g}})\subseteq [\mathfrak{g},\mathfrak{g}]\bigotimes\mathcal{V}_{\mathfrak{g}}^{\ast},\quad
\mathrm{d}([\mathfrak{g},\mathfrak{g}])=\mathrm{d}(\mathcal{V}_{\mathfrak{g}}^{\ast})=0.
\end{equation*}
Set
$$\mathcal{W}_{\mathfrak{g}}^{k}=\mathcal{V}_{\mathfrak{g}}\bigotimes\bigwedge^{k}\mathcal{V}_{\mathfrak{g}}^{\ast},\quad
\mathcal{\overline{W}}_{\mathfrak{g}}^{k,i}=[\mathfrak{g},\mathfrak{g}]\bigotimes\left(\bigwedge^{i}[\mathfrak{g},\mathfrak{g}]^{\ast}\bigwedge^{k-i} \mathcal{V}_{\mathfrak{g}}^{\ast}\right).$$
Moreover, from the conditions $(1)$ and $(2)$,  $\mathrm{H}^{k}(\mathfrak{g},\mathfrak{g})$ is contained in
$\mathrm{Ker}\ \mathrm{d}\bigcap \left(\mathcal{W}_{\mathfrak{g}}^{k}\bigoplus \mathcal{\overline{W}}_{\mathfrak{g}}^{k,1}\right)$.
From Theorem \ref{cup}, we have
$\mathcal{\overline{W}}_{\mathfrak{g}}^{k,1}\bigcup\mathrm{H}^{\bullet}(\mathfrak{g},\mathfrak{g})=0$
and
$$\mathcal{W}_{\mathfrak{g}}^{k}\bigcup\mathcal{W}_{\mathfrak{g}}^{l}\subseteq \mathcal{\overline{W}}_{\mathfrak{g}}^{k+l,0}\subseteq \mathrm{Im}\ \mathrm{d},$$
for $k,l\geq 1$. The proof is complete.
\end{proof}

For a
Lie superalgebra, denote $k$-th Betti number by the dimension of the $k$-cohomology.
We are in position to study the Betti numbers of adjoint cohomology  for the two-step nilipotent Lie superalgebras.
A useful tool to compute the Betti numbers is spectral sequence (see \cite{9,10,11} for example). Suppose $I$ is an abelian ideal of $\mathfrak{g}$. Then there is a convergent spectral sequence $\{\mathrm{E}_{r}^{p,q},\mathrm{d}_{r}^{p,q}:
\mathrm{E}_{r}^{p,q}\rightarrow\mathrm{E}_{r}^{p+r,q-r+1}\}$,
called the \emph{Hochschild-Serre spectral sequence}, such that
\begin{align*}
\mathrm{E}_{2}^{k,s}=\mathrm{H}^{k}(\mathfrak{g}/I,\mathrm{H}^{s}(I,\mathfrak{g}))\Longrightarrow \mathrm{H}^{k+s}(\mathfrak{g},\mathfrak{g}),
\end{align*}
(see \cite[Theorem 16.6.6]{6}). Suppose that $I=C(\mathfrak{g})$. Then
\begin{align}\label{26}
\mathrm{E}_{2}^{k,s}=\mathrm{H}^{k}(\mathfrak{g}/C(\mathfrak{g}),\mathfrak{g})\bigotimes \bigwedge^{s}C(\mathfrak{g})^{\ast}\Longrightarrow \mathrm{H}^{k+s}(\mathfrak{g},\mathfrak{g}).
\end{align}
\begin{rem}
It shows that Theorem \ref{main} is only a sufficient condition of the cup product triviality, not a necessary condition.
\end{rem}
\begin{thm}\label{main1}
Suppose that $\mathfrak{g}$ is a two-step nilpotent Lie superalgebra and $\mathcal{\overline{V}}_{\mathfrak{g}}$ is any complement space of the center $C(\mathfrak{g})$ in $\mathfrak{g}$. Let $\mathcal{\overline{V}}_{\mathfrak{g}}^{\ast}$ be the dual space of $\mathcal{\overline{V}}_{\mathfrak{g}}$.

For $k\geq 1$, if the following two conditions are satisfied:

(1) $C(\mathfrak{g})\bigotimes \bigwedge^{k}\mathcal{\overline{V}}_{\mathfrak{g}}^{\ast}\subseteq \mathrm{Im}\ \mathrm{d}$,

(2) For any $f\in C(\mathfrak{g})^{\ast}$ and $g\in\bigwedge^{k}\mathcal{\overline{V}}_{\mathfrak{g}}^{\ast}$, if $\mathrm{d}(f)\wedge g=0$, then $f=0$ or $g=0$, \\
then $\mathrm{Ker}\ \mathrm{d}_{2}^{k,1}=0$. In particular, $\mathrm{E}_{\infty}^{k,1}=0$.
\end{thm}
\begin{proof}
From Eq. (\ref{26}), for $k\geq 1$, consider the mapping
\begin{align*}
\mathrm{d}_{2}^{k,1}: \mathrm{H}^{k}(\mathfrak{g}/C(\mathfrak{g}),\mathfrak{g})\bigotimes C(\mathfrak{g})^{\ast}&\longrightarrow
\mathrm{H}^{k+2}(\mathfrak{g}/C(\mathfrak{g}),\mathfrak{g}),
\\
 f\otimes g&\longmapsto f\wedge \mathrm{d} (g),
\end{align*}
where $f\in\mathrm{H}^{k}(\mathfrak{g}/C(\mathfrak{g}),\mathfrak{g})$, and $g\in C(\mathfrak{g})^{\ast}$.
Since $\mathfrak{g}/C(\mathfrak{g})$ is abelian, from the condition $(1)$,
$$\mathrm{H}^{k}(\mathfrak{g}/C(\mathfrak{g}),\mathfrak{g})=\mathrm{Ker}\ \mathrm{d}\bigcap\left\{\mathcal{\overline{V}}_{\mathfrak{g}}\bigotimes\bigwedge^{k}\left(\mathfrak{g}/C(\mathfrak{g})\right)^{\ast}\right \}.$$
Define the linear mapping
\begin{align*}
\widetilde{\mathrm{d}_{2}^{k,1}}: \left\{\mathcal{\overline{V}}_{\mathfrak{g}}\bigotimes\bigwedge^{k} (\mathfrak{g}/C(\mathfrak{g}))^{\ast}\right\}\bigotimes C(\mathfrak{g})^{\ast}&\longrightarrow
\mathcal{\overline{V}}_{\mathfrak{g}}\bigotimes\bigwedge^{k+2} (\mathfrak{g}/C(\mathfrak{g}))^{\ast},
\\
 f\otimes g&\longmapsto  f\wedge \mathrm{d}( g),
\end{align*}
where $f\in \mathcal{\overline{V}}_{\mathfrak{g}}\bigotimes\bigwedge^{k} (\mathfrak{g}/C(\mathfrak{g}))^{\ast}$, and $g\in C(\mathfrak{g})^{\ast}$. Then, from the condition $(2)$,
$$\mathrm{Ker}\ \mathrm{d}_{2}^{k,1}\subseteq \mathrm{Ker}\ \widetilde{ \mathrm{d}_{2}^{k,1}}=0.$$
Moreover, $\mathrm{E}_{\infty}^{k,1}=\mathrm{E}_{3}^{k,1}=0.$
\end{proof}

\section{Heisenberg Lie superalgebra}
If $\mathfrak{g}$ is a two-step nilpotent Lie superalgebra with a one-dimensional center, $\mathfrak{g}$ is called a \emph{Heisenberg Lie superalgebra}. All finite dimensional Heisenberg Lie superalgebras are divided into two classes, according to the parities of their centers, denoted by $\mathfrak{h}_{2m,n}$ and $\mathfrak{ba}_{n}$ (see \cite{19}).

$(1)\ \mathfrak{h}_{2m,n}$ has a homogeneous basis
$$\{z;x_{1},\ldots,x_{m},x_{m+1},\ldots,x_{2m}\mid y_{1},\ldots,y_{n}\},$$
where $|z|=|x_{1}|=\ldots=|x_{2m}|=\bar{0}$, $|y_{1}|=\ldots=|y_{n}|=\bar{1}$
and the non-zero brackets are given by
$$[x_{i},x_{m+i}]=[y_{j},y_{j}]=z, \quad 1\leq i\leq m\ \mathrm{and}\ 1\leq j\leq n.$$

$(2)\ \mathfrak{ba}_{n}$ has a homogeneous basis
$$\{x_{1},\ldots,x_{n}\mid z;y_{1},\ldots,y_{n}\},$$
where $|x_{1}|=\ldots=|x_{n}|=\bar{0}$, $|z|=|y_{1}|=\ldots=|y_{n}|=\bar{1}$
and the non-zero brackets are given by
$$[x_{i},y_{i}]=z, \quad 1\leq i\leq n.$$
Suppose that $\mathfrak{g}$ is a Heisenberg Lie superalgebra. Note that
$C(\mathfrak{g})=[\mathfrak{g},\mathfrak{g}]=\F z$. Then
$\mathrm{H}^{0}(\mathfrak{g},\mathfrak{g})=\F z.$
As an application of Theorems \ref{main} and \ref{main1}, we characterize below the cup products and Betti numbers of adjoint cohomology of Heisenberg Lie superalgebras. For a Heisenberg Lie superalgebra $\mathfrak{g}$,
we take
\begin{align*}
\mathcal{V}_{\mathfrak{g}}=\mathcal{\overline{V}}_{\mathfrak{g}}&=\left\{
                       \begin{array}{ll}
                        \mathrm{span}_{\F}\{x_{1},\ldots,x_{2m}\mid y_{1},\ldots,y_{n}\}, & \hbox{$\mathfrak{g}=\mathfrak{h}_{2m,n};$} \\
                        \mathrm{span}_{\F}\{x_{1},\ldots,x_{n}\mid y_{1},\ldots,y_{n}\}, & \hbox{$\mathfrak{g}=\mathfrak{ba}_{n},$}
                        \end{array}
                       \right.
\end{align*}
such that $\mathfrak{g}=[\mathfrak{g},\mathfrak{g}]\bigoplus \mathcal{V}_{\mathfrak{g}}=C(\mathfrak{g})\bigoplus \mathcal{\overline{V}}_{\mathfrak{g}}$.
At first,
from Eq. (\ref{z}), we have
\begin{align*}
(\mathrm{d}z^{\ast})_{\mathfrak{g}}&=\left\{
                       \begin{array}{ll}
                        -\sum\limits_{i=1}^{m}x_{i}^{\ast}\wedge x_{m+i}^{\ast}+\frac{1}{2}\sum\limits_{j=1}^{n}y_{j}^{\ast^{2}}, & \hbox{$\mathfrak{g}=\mathfrak{h}_{2m,n};$} \\
                        -\sum\limits_{i=1}^{n}x_{i}^{\ast}\wedge y_{i}^{\ast}, & \hbox{$\mathfrak{g}=\mathfrak{ba}_{n}.$}
                        \end{array}
                       \right.
\end{align*}
For $k\geq 0$,
set
\begin{equation*}
\psi_{\mathfrak{g}}^{k}:\bigwedge^{k}(\mathfrak{g}/\F z)^{\ast}\longrightarrow\bigwedge^{k+2}(\mathfrak{g}/\F z)^{\ast},\quad \alpha\mapsto \alpha\wedge (\mathrm{d}z^{\ast})_{\mathfrak{g}}.
\end{equation*}
From \cite[Lemma 4.3]{9} and the proofs of \cite[Theorems 4.1 and 4.2]{9}, we obtain the following lemma:
\begin{lem}\label{bai}
Suppose that $k\geq 0$ and $\mathfrak{g}$ is a Heisenberg Lie superalgebra.
Then
\begin{align*}
\mathrm{Ker}\ \psi_{\mathfrak{g}}^{k}&=\left\{
                       \begin{array}{ll}
                        0, & \hbox{$\mathfrak{g}=\mathfrak{h}_{2m,n};$} \\
                        \bigwedge^{k-2}(\mathfrak{ba}_{n}/\F z)^{\ast}\bigwedge \F (\mathrm{d}z^{\ast})_{\mathfrak{g}}\bigoplus \delta_{k,n}\F (x_{1}^{\ast}\wedge\cdots\wedge x_{n}^{\ast}), & \hbox{$\mathfrak{g}=\mathfrak{ba}_{n}.$}
                        \end{array}
                       \right.
\end{align*}
Moreover,
\begin{align*}
\mathrm{dim}\ \mathrm{Ker}\ \psi_{\mathfrak{g}}^{k}&=\left\{
                       \begin{array}{ll}
                        0, & \hbox{$\mathfrak{g}=\mathfrak{h}_{2m,n};$} \\
                        \sum\limits_{i=1}^{\lfloor\frac{k}{2}\rfloor}(-1)^{i-1}(\mathfrak{d}^{k-2i}(n,n)-\delta_{k-2i,n})+\delta_{k,n}, & \hbox{$\mathfrak{g}=\mathfrak{ba}_{n}.$}
                        \end{array}
                       \right.
\end{align*}
\end{lem}

\begin{thm}
The cup product on the adjoint cohomology for Heisenberg Lie superalgebras is trivial.
\begin{proof}
From Theorem \ref{main} and Lemma \ref{bai}, it is obvious that the theorem holds for $\mathfrak{h}_{2m,n}$.
Suppose that $k\geq 1$ and $\mathfrak{g}=\mathfrak{ba}_{n}$. For $1\leq i\leq k$, set
$$
\mathcal{W}_{\mathfrak{g}}^{k,i}=\mathcal{V}_{\mathfrak{g}}\bigotimes\left(\F z^{\ast^{i}}\bigwedge \mathrm{Ker}\ \psi_{\mathfrak{g}}^{k-i}\right).$$
By a direct computation from Eqs. (\ref{z}) and (\ref{22}), $\mathrm{H}^{k}(\mathfrak{g},\mathfrak{g})$ is contained in
$$\left\{\mathrm{Ker}\ \mathrm{d}\bigcap\left(\mathcal{W}_{\mathfrak{g}}^{k}\bigoplus\mathcal{\overline{W}}_{\mathfrak{g}}^{k,1}\right)\right\}
\bigoplus_{i=2}^{k}\left\{\mathrm{Ker}\ \mathrm{d}\bigcap\left(\mathcal{W}_{\mathfrak{g}}^{k,i-1}\bigoplus\mathcal{\overline{W}}_{\mathfrak{g}}^{k,i}\right)\right\}.$$
Note that
\begin{equation*}
C(\mathfrak{g})=[\mathfrak{g},\mathfrak{g}]=\F z,\quad
\mathrm{d}\left(\mathcal{V}_{\mathfrak{g}}\bigotimes\bigwedge^{k-1}\mathcal{V}_{\mathfrak{g}}^{\ast}\right)=\F z\bigotimes\bigwedge^{k}\mathcal{V}_{\mathfrak{g}}^{\ast},  \quad
 \mathrm{Ker}\ \psi_{\mathfrak{g}}^{k}\bigwedge\mathrm{Ker}\ \psi_{\mathfrak{g}}^{l}=0,
\end{equation*}
where $k,l\geq 1$.
Thus,
from Theorem \ref{cup},
$$\mathcal{W}_{\mathfrak{g}}^{k}\bigcup \mathcal{W}_{\mathfrak{g}}^{l}=\mathcal{W}_{\mathfrak{g}}^{k,i}\bigcup\mathcal{W}_{\mathfrak{g}}^{l,j}
=\mathcal{\overline{W}}_{\mathfrak{g}}^{k,i}\bigcup\mathrm{H}^{\bullet}(\mathfrak{g},\mathfrak{g})=0,$$
for any $k,l\geq 1$, $1\leq i\leq k$, and $1\leq j\leq l$.
It is sufficient to prove that
$$\mathcal{W}_{\mathfrak{g}}^{k}\bigcup\mathcal{W}_{\mathfrak{g}}^{l,j}=0.$$
From Lemma \ref{bai} and  Theorem \ref{cup}, we have
\begin{align*}
\mathcal{W}_{\mathfrak{g}}^{k}\bigcup\mathcal{W}_{\mathfrak{g}}^{l,j}&\subseteq \F z\bigotimes\F z^{\ast^{j}}\left(\bigwedge^{k}\mathcal{V}_{\mathfrak{g}}^{\ast}\bigwedge\mathrm{Ker}\ \psi_{\mathfrak{g}}^{l-j}\right)\\
& \subseteq \F z\bigotimes\F \left((\mathrm{d}z^{\ast})_{\mathfrak{g}}\wedge z^{\ast^{j}}\right)\bigwedge^{k+l-j-2}\mathcal{V}_{\mathfrak{g}}^{\ast}\\
& \subseteq \mathrm{d}\left(\F z\bigotimes\F z^{\ast^{j+1}}\bigwedge^{k+l-j-2}\mathcal{V}_{\mathfrak{g}}^{\ast}\right).
\end{align*}
The proof is complete.
\end{proof}
\end{thm}
\begin{lem}\label{q}
Suppose that $\mathfrak{g}=\F z \bigoplus \mathcal{V}_{\mathfrak{g}}$ is a Heisenberg Lie superalgebra and the super-dimension of $\mathcal{V}_{\mathfrak{g}}$ is $(r,s)$.
Then
$$\mathrm{dim}\ \mathrm{H}^{k}(\mathfrak{g}/\F z,\mathfrak{g})=
                        (r+s)\mathfrak{d}^{k}(r,s)-\mathfrak{d}^{k+1}(r,s),
$$
for $k\geq 1$.
\end{lem}
\begin{proof}
By a direct computation, we have
\begin{equation*}
\mathrm{d}\left(\mathfrak{g} \bigotimes\bigwedge^{k-1}(\mathfrak{g}/ \F z)^{\ast}\right)=\mathrm{d}\left(\mathcal{V}_{\mathfrak{g}} \bigotimes\bigwedge^{k-1}(\mathfrak{g}/ \F z)^{\ast}\right)=\F z\bigotimes\bigwedge^{k}\left(\mathfrak{g}/ \F z\right)^{\ast},
\end{equation*}
and
\begin{equation*}
\mathrm{H}^{k}(\mathfrak{g}/\F z,\mathfrak{g})=\mathrm{Ker}\ \mathrm{d}\bigcap \left(\mathcal{V}_{\mathfrak{g}}\bigotimes\bigwedge^{k} (\mathfrak{g}/\F z)^{\ast}\right),
\end{equation*}
where $k\geq 1$. It is sufficient to consider the dimension of
$$\mathrm{Ker}\ \mathrm{d}\bigcap \left(\mathcal{V}_{\mathfrak{g}}\bigotimes\bigwedge^{k} (\mathfrak{g}/\F z)^{\ast}\right),\quad k\geq 1.$$
For $k\geq 1$, define the mapping
$$\overline{\mathrm{d}}_{k}: \mathcal{V}_{\mathfrak{g}}\bigotimes\bigwedge^{k} (\mathfrak{g}/\F z)^{\ast}\longrightarrow \F z \bigotimes \bigwedge^{k+1}(\mathfrak{g}/\F z)^{\ast}
$$
such that $\overline{\mathrm{d}}_{k}(x)=\mathrm{d}(x)$, $x\in\mathcal{V}_{\mathfrak{g}}\bigotimes\bigwedge^{k} (\mathfrak{g}/\F z)^{\ast}$.
Then $\overline{\mathrm{d}}_{k}$ is surjective. Moreover, we have
\begin{align*}
\mathrm{dim}\ \mathrm{Ker}\ \mathrm{d}\bigcap \left(\mathcal{V}_{\mathfrak{g}}\bigotimes\bigwedge^{k} (\mathfrak{g}/\F z)^{\ast}\right)
=&\mathrm{dim}\ \mathrm{Ker}\ \overline{\mathrm{d}}_{k}\\
=&\mathrm{dim}\ \mathcal{V}_{\mathfrak{g}}\bigotimes\bigwedge^{k} (\mathfrak{g}/\F z)^{\ast}-\mathrm{dim}\ \F z\bigotimes\bigwedge^{k+1}(\mathfrak{g}/\F z)^{\ast}\\
=&(r+s)\mathfrak{d}^{k}(r,s)-\mathfrak{d}^{k+1}(r,s),
\end{align*}
the proof is complete.
\end{proof}

For $k,n\in \mathbb{Z}$, set
$$\mathfrak{a}_{n}^{k}=\mathrm{dim}\ \mathrm{Ker}\ \psi_{\mathfrak{ba}_{n}}^{k}$$ and
$$\mathfrak{b}_{n}^{k}=\mathfrak{d}^{k}(n,n)-\mathfrak{a}_{n}^{k}.$$
\begin{thm}\label{even}
Suppose that $k\geq 1$.
Then
\begin{align*}
(1)&\ \mathrm{dim}\ \mathrm{H}^{k}(\mathfrak{h}_{2m,n},\mathfrak{h}_{2m,n})=
                        (2m+n)(\mathfrak{d}^{k}(2m,n)-\mathfrak{d}^{k-2}(2m,n))-\mathfrak{d}^{k+1}(2m,n)\\
&\qquad\qquad\qquad\qquad\quad\quad\quad +\mathfrak{d}^{k-1}(2m,n).\\
(2)&\ \mathrm{dim}\ \mathrm{H}^{k}(\mathfrak{ba}_{n},\mathfrak{ba}_{n})
=2n\mathfrak{d}^{k}(n,n)-\mathfrak{d}^{k+1}(n,n)+1+\sum_{i=1}^{k-1}(2n\mathfrak{a}_{n}^{k-i}-\mathfrak{b}_{n}^{k-i-1}) \\
&\qquad\qquad\qquad\qquad\quad+\sum_{i=0}^{k-3}
(\mathfrak{d}^{k-i-1}(n,n)-\mathfrak{b}_{n}^{k-i-3}-2n\mathfrak{b}_{n}^{k-i-2}).
\end{align*}
\end{thm}
\begin{proof}
Suppose that $\mathfrak{g}$ is a Heisenberg Lie superalgebra.
Note that the center $C(\mathfrak{g})=\F z$. From Eq. (\ref{26}), for $k\geq 1$, and $0\leq i\leq k$, we have
\begin{equation}\label{3.2}
\mathrm{E}_{2}^{k-i,i}=\mathrm{H}^{k-i}(\mathfrak{g}/\F z,\mathfrak{g})\bigotimes\F z^{\ast^{i}}.
\end{equation}
Moreover, $\mathrm{E}_{\infty}^{k-i,i}=\mathrm{E}_{3}^{k-i,i}.$ Since $\mathrm{H}^{0}(\mathfrak{g}/\F z,\mathfrak{g})=\F z$, consider the mapping:
\begin{align*}
\mathrm{d}_{2}^{0,i}: \F (z\otimes z^{\ast^{i}})&\longrightarrow
\mathrm{H}^{2}(\mathfrak{g}/\F z,\mathfrak{g})\bigotimes\F z^{\ast^{i-1}},
\\
z\otimes z^{\ast^{i}}&\longmapsto iz\otimes (\mathrm{d} z^{\ast})_{\mathfrak{g}}\otimes  z^{\ast^{i-1}},
\end{align*}
where $i\geq 0$. Note that $z\otimes (\mathrm{d} z^{\ast})_{\mathfrak{g}}\in \mathrm{d}(\mathfrak{g}\bigotimes (\mathfrak{g}/\F z)^{\ast})$, which is zero in $\mathrm{H}^{2}(\mathfrak{g}/\F z,\mathfrak{g})$. Thus, we have
 \begin{equation}\label{3.3}
\mathrm{E}_{\infty}^{0,i}=\mathrm{E}_{3}^{0,i}=\mathrm{Ker}\ \mathrm{d}_{2}^{0,i}=\F \left(z\otimes z^{\ast^{i}}\right),
\end{equation}
for $i\geq 0.$

(1)
Suppose that $\mathfrak{g}=\mathfrak{h}_{2m,n}$. Note that $|z|=\overline{0}$. From Eq. (\ref{3.2}), for $k\geq 1$, we have
\begin{equation*}
\mathrm{E}_{2}^{k-i,i}
                   =\left\{
                       \begin{array}{ll}
                        \mathrm{H}^{k}(\mathfrak{h}_{2m,n}/\F z,\mathfrak{h}_{2m,n}), & \hbox{$i=0;$} \\
                        \mathrm{H}^{k-1}(\mathfrak{h}_{2m,n}/\F z,\mathfrak{h}_{2m,n})\bigotimes\F z^{\ast}, & \hbox{$i=1;$} \\
                        0, & \hbox{$i\geq2.$}
                        \end{array}
                       \right.
\end{equation*}
From Theorem \ref{main1} and Eq. (\ref{3.3}),
\begin{equation*}
\mathrm{E}_{\infty}^{k-i,i}=\left\{
                       \begin{array}{ll}
                        \mathrm{H}^{k}(\mathfrak{h}_{2m,n}/\F z,\mathfrak{h}_{2m,n})/ \mathrm{Im}\ \mathrm{d}_{2}^{k-2,1}, & \hbox{$i=0,\ k\geq 1;$} \\
                       \F \left(z\otimes z^{\ast}\right) , & \hbox{$i=k=1;$} \\
                        0, & \hbox{$i=1,\ k \geq 2;$ or $i\geq 2.$}
                        \end{array}
                       \right.
\end{equation*}
Moreover, we have
\begin{eqnarray*}
&&\mathrm{dim}\ \mathrm{H}^{1}(\mathfrak{h}_{2m,n},\mathfrak{h}_{2m,n})=\mathrm{dim}\ \mathrm{H}^{1}(\mathfrak{h}_{2m,n}/\F z,\mathfrak{h}_{2m,n})+1, \\
&&\mathrm{dim}\ \mathrm{H}^{2}(\mathfrak{h}_{2m,n},\mathfrak{h}_{2m,n})=\mathrm{dim}\ \mathrm{H}^{2}(\mathfrak{h}_{2m,n}/\F z,\mathfrak{h}_{2m,n}),
\end{eqnarray*}
and
\begin{equation*}
\mathrm{dim}\ \mathrm{H}^{k}(\mathfrak{h}_{2m,n},\mathfrak{h}_{2m,n})=\mathrm{dim}\ \mathrm{H}^{k}(\mathfrak{h}_{2m,n}/\F z,\mathfrak{h}_{2m,n})-\mathrm{dim}\ \mathrm{H}^{k-2}(\mathfrak{h}_{2m,n}/\F z,\mathfrak{h}_{2m,n}),
\end{equation*}
where $k\geq 3$.
Thanks to Lemma \ref{q}, the proof is complete.

(2)
Suppose that $\mathfrak{g}=\mathfrak{ba}_{n}$.
From the following sequence
\begin{align*}
&\mathrm{E}_{2}^{k-i-2,i+1}\stackrel{\mathrm{d}_{2}^{k-i-2,i+1}}{\longrightarrow} \mathrm{E}_{2}^{k-i,i}\stackrel{\mathrm{d}_{2}^{k-i,i}}{\longrightarrow} \mathrm{E}_{2}^{k-i+2,i-1},
\end{align*}
we have
\begin{equation*}
\mathrm{E}_{\infty}^{k-i,i}=\mathrm{Ker}\ \mathrm{d}_{2}^{k-i,i}/\mathrm{Im}\ \mathrm{d}_{2}^{k-i-2,i+1}.
\end{equation*}
Moreover, from Eq. (\ref{3.2}),
\begin{equation}\label{2.2.1}
\mathrm{dim}\ \mathrm{H}^{k}(\mathfrak{ba}_{n}, \mathfrak{ba}_{n})=\sum\limits_{i=0}^{k} \left(\mathrm{dim}\ \mathrm{Ker}\ \mathrm{d}_{2}^{k-i,i}+\mathrm{dim}\ \mathrm{Ker}\ \mathrm{d}_{2}^{k-i-2,i+1}
-\mathrm{dim}\ \mathrm{H}^{k-i-2}(\mathfrak{ba}_{n}/\F z,\mathfrak{ba}_{n})\right).
\end{equation}
Consider the linear mapping:
\begin{align*}
\mathrm{d}_{2}^{k-i,i}: \mathrm{H}^{k-i}(\mathfrak{ba}_{n}/\F z,\mathfrak{ba}_{n})\bigotimes\F z^{\ast^{i}}&\longrightarrow
\mathrm{H}^{k-i+2}(\mathfrak{ba}_{n}/\F z,\mathfrak{ba}_{n})\bigotimes\F z^{\ast^{i-1}},
\\
x\otimes z^{\ast^{i}}&\longmapsto ix\wedge (\mathrm{d} z^{\ast})_{\mathfrak{ba}_{n}}\otimes z^{\ast^{i-1}},
\end{align*}
where $x\in \mathrm{H}^{k-i}(\mathfrak{ba}_{n}/\F z,\mathfrak{ba}_{n})$.
In order to compute the dimension of $\mathrm{Ker}\ \mathrm{d}_{2}^{k-i,i}$,
we define the linear mapping
\begin{align*}
\widetilde{\mathrm{d}_{2}^{k-i,i}}: \left(\mathcal{V}_{\mathfrak{ba}_{n}}\bigotimes\bigwedge^{k-i} (\mathfrak{ba}_{n}/\F z)^{\ast}\right)\bigotimes \F z^{\ast^{i}}&\longrightarrow
\left(\mathcal{V}_{\mathfrak{ba}_{n}}\bigotimes\bigwedge^{k-i+2} (\mathfrak{ba}_{n}/\F z)^{\ast}\right)\bigotimes\F z^{\ast^{i-1}},
\\
 x\otimes z^{\ast^{i}}&\longmapsto ix\wedge (\mathrm{d} z^{\ast})_{\mathfrak{ba}_{n}}\otimes z^{\ast^{i-1}},
\end{align*}
where $k\geq1$, $1\leq i\leq k-1$ and $x\in\mathcal{V}_{\mathfrak{ba}_{n}}\bigotimes\bigwedge^{k-i} (\mathfrak{ba}_{n}/\F z)^{\ast}$.
 Then
$$\mathrm{Ker}\ \widetilde{\mathrm{d}_{2}^{k-i,i}}=\mathcal{V}_{\mathfrak{ba}_{n}}\bigotimes \mathrm{Ker}\ \psi_{\mathfrak{ba}_{n}}^{k-i}\bigotimes \F z^{\ast^{i}}.$$
Set
$$
\widetilde{\mathcal{V}}_{k,i}=\mathcal{V}_{\mathfrak{ba}_{n}}\bigotimes \mathrm{Ker}\ \psi_{\mathfrak{ba}_{n}}^{k-i}.$$
Then, for $k\geq 1$,
\begin{equation*}
\mathrm{Ker}\ \mathrm{d}_{2}^{k-i,i}=\left\{
                       \begin{array}{ll}
                        \mathrm{H}^{k}(\mathfrak{ba}_{n}/\F z,\mathfrak{ba}_{n}), & \hbox{$i=0;$} \\
 \left(\mathrm{Ker}\ \mathrm{d}\bigcap \widetilde{\mathcal{V}}_{k,i}\right)\bigotimes \F z^{\ast^{i}} , & \hbox{$1\leq i\leq k-1;$}\\
\F \left(z\otimes z^{\ast^{k}}\right) , & \hbox{$i=k.$}
                        \end{array}
                       \right.
\end{equation*}
To compute the dimension of $\mathrm{Ker}\ \mathrm{d}\bigcap \widetilde{\mathcal{V}}_{k,i}$,
for $k\geq 1$ and $1\leq i\leq k-1$, we set
\begin{align*}
&\overline{\mathrm{d}}_{k,i}: \widetilde{\mathcal{V}}_{k,i}\longrightarrow \F z \bigotimes \left(\bigwedge^{k-i-1}(\mathfrak{ba}_{n}/\F z)^{\ast}\bigwedge \F (\mathrm{d}z^{\ast})_{\mathfrak{ba}_{n}}\right)
\end{align*}
such that $\overline{\mathrm{d}}_{k,i}(x)=\mathrm{d}(x)$, $x\in\widetilde{\mathcal{V}}_{k,i}$.
Then, from Lemma \ref{bai}, $\overline{\mathrm{d}}_{k,i}$ is surjective. Moreover, we have
\begin{align*}
\mathrm{dim}\ \mathrm{Ker}\ \mathrm{d}\bigcap  \widetilde{\mathcal{V}}_{k,i}
=&\mathrm{dim}\ \mathrm{Ker}\ \overline{\mathrm{d}}_{k,i}\\
=&\mathrm{dim}\  \widetilde{\mathcal{V}}_{k,i}-
\mathrm{dim}\ \bigwedge^{k-i-1}(\mathfrak{ba}_{n}/\F z)^{\ast}\bigwedge \F (\mathrm{d}z^{\ast})_{\mathfrak{ba}_{n}}\\
=&2n\mathrm{dim}\ \mathrm{Ker}\ \psi_{\mathfrak{ba}_{n}}^{k-i}-\mathfrak{d}^{k-i-1}(n,n)+\mathrm{dim}\ \mathrm{Ker}\ \psi_{\mathfrak{ba}_{n}}^{k-i-1}.
\end{align*}
From Lemmas \ref{bai} and \ref{q} and Eq. (\ref{2.2.1}),
the proof is complete.
\end{proof}

\small\noindent \textbf{Acknowledgment}\\
The first author was supported by the NSF of China (11471090). The second author was supported by Graduate Innovation Fund of Jilin University (101832018C161). The third author was supported by the NSF of China (11771176).

\end{document}